\documentclass[amssymb,amsfonts,refcheck,12pt,verbatim,righttag]{amsart}
\usepackage{amsfonts}

 \numberwithin{equation}{section}
\theoremstyle{plain}

\newtheorem{thm}{Theorem}[section]
\newtheorem{lem}[thm]{Lemma}
\newtheorem{pro}[thm]{Proposition}

\newtheorem{de}[thm]{Definition}
\newtheorem{rem}[thm]{Remark}

\newcommand{\ber}{\nu_{\lambda}}

\def\R {{\Bbb R}}
\def\N {{\Bbb N}}
\def\Z {{\Bbb Z}}

\def\D {{\mathcal D}}

\begin{document}
\baselineskip 16pt

\title[Multifractal analysis of Bernoulli convolutions]{Multifractal analysis of  Bernoulli convolutions associated with Salem numbers}

\date{}

\author{ De-Jun FENG}
\address{
Department of Mathematics,
The Chinese University of Hong Kong,
Shatin,  Hong Kong}

\email{djfeng@math.cuhk.edu.hk}



\maketitle

\begin{abstract}
We consider the multifractal structure of the  Bernoulli convolution $\nu_{\lambda}$, where $\lambda^{-1}$ is a Salem number in $(1,2)$.
Let $\tau(q)$ denote the $L^q$ spectrum of $\nu_\lambda$. We show that if $\alpha \in [\tau'(+\infty), \tau'(0+)]$, then
the level set  $$E(\alpha):=\left\{x\in \R:\; \lim_{r\to 0}\frac{\log \nu_\lambda([x-r, x+r])}{\log r}=\alpha\right\}$$
is non-empty and $\dim_HE(\alpha)=\tau^*(\alpha)$, where $\tau^*$ denotes the Legendre transform of $\tau$.
This result extends to  all self-conformal measures  satisfying the asymptotically weak separation condition.
We point out that the interval $[\tau'(+\infty), \tau'(0+)]$
is  not  a singleton when $\lambda^{-1}$ is the largest real root of the polynomial $x^{n}-x^{n-1}-\cdots -x+1$,  $n\geq 4$. An example is constructed to show that
absolutely continuous self-similar measures may also have rich multifractal structures.
\end{abstract}
\section{Introduction}
\label{S-1}

For any $\lambda\in (0,1)$, let $\ber$ denote the distribution of
$\sum_{n=0}^\infty \epsilon_n \lambda^n$
where the coefficients $\epsilon_n$ are either $-1$ or $1$,
chosen independently with probability $\frac{1}{2}$ for each.
It is the infinite convolution product of the distributions
$\frac{1}{2}(\delta_{-\lambda^n}+\delta_{\lambda^n})$, giving rise to the term
``infinite
Bernoulli convolution'' or simply ``Bernoulli convolution''.  The Bernoulli
convolution
can be expressed as a self-similar measure $\ber$ satisfying the equation
\begin{equation} \label{1.1}
    \ber = \frac{1}{2}\ber\circ S^{-1}_1 +\frac{1}{2}\ber\circ S^{-1}_2,
\end{equation}
where $S_1(x) = \lambda x-1$ and $S_2(x)=\lambda x+ 1$.
These measures have been studied since the 1930's, revealing  surprising
connections with a number of areas in mathematics, such as harmonic
analysis, fractal geometry, number theory, dynamical systems, and others, see \cite{PSS00}.

The fundamental question about $\ber$  is to decide for which  $\lambda\in (\frac12,1)$ this measure is absolutely continuous and for which $\lambda$ it is singular.  It is well known that
for each $\lambda\in (1/2, 1)$,  $\ber$ is continuous, and  it is either purely absolutely continuous or purely
singular.  Solomyak \cite{Sol95} proved that $\ber$ is absolutely continuous for a.e. $\lambda\in (1/2, 1)$.
In the other direction,   Erd\"os
\cite{Erd39} proved that
if $\lambda^{-1}$ is a Pisot number, i.e. an algebraic integer whose
algebraic conjugates
are all inside the unit disk, then $\ber$ is singular. It is an open problem whether the Pisot reciprocals are the only
class of $\lambda$'s in $(\frac{1}{2}, 1)$ for which $\ber$ is singular.
This question is far from being answered.
There appears to be a general belief that the
best candidates for counter-examples are
the reciprocals of {\em Salem numbers}. Recall that a positive number $\beta$ is called a Salem number
if it
is an algebraic integer whose algebraic conjugates all have modulus no
greater than 1,
with at least one of which on the unit circle. Indeed, as Kahane observed, when $\lambda^{-1}$ is a Salem number,
the Fourier transform of $\ber$ has no uniform decay at infinity (cf. \cite[Lemma 5.2]{PSS00}).  A well-known class of Salem numbers are the largest real roots $\beta_n$ of the polynomials
$x^n-x^{n-1}-\cdots-x+1$; where $n\geq 4$.   It was  shown by   Wang and  the author in \cite{FeWa04} that for any $\epsilon>0$, the density of $\nu_{1/\beta_n}$, if it exists, is not in
$L^{3+\epsilon}(\R)$ when $n$ is large enough.

In this paper, we study the local dimensions and the multifractal structure of $\nu_\lambda$ when $\lambda^{-1}$ is a Salem number in $(1,2)$. Few results along this direction have been known in the literature.  Before formulating our results, we first recall some basic notation used in the multifractal analysis. The reader is referred to \cite{Fal-book} for details.

Let $\mu$ be a finite  Borel measure in $\R^d$ with compact support. For $x\in \R^d$ and $r>0$, let $B_r(x)$ denote the closed ball centered at $x$ of radius $r$.  For $q\in \R$,  the {\it $L^q$ spectrum} of $\mu$ is defined as
$$
\tau_\mu(q)=\liminf_{r\rightarrow 0}\frac{\log \Theta_\mu(q;r)}{\log r},
$$
where
\begin{equation}
\label{e-tz1}
\Theta_\mu(q;r)= \sup
\sum_{i}\mu (B_r(x_i))^q ,\qquad r>0,\; q\in \R,
\end{equation}
and the supremum is taken over all  families of disjoint balls $%
\{B_r(x_i)\}_{i}$  with $x_{i}\in \mbox{supp}(\mu)$.
It is easily checked that ${\tau}_\mu(q)$ is a concave
function of $q$ over $\R$. For $x\in \R^d$, the {\it local dimension} of $\mu$ at $x$ is defined as
$$
d_\mu(x)=\lim_{r\to 0}\frac{\log \mu(B_r(x))}{\log r},
$$
provided that the limit exists.
For $\alpha\in \R$, denote
$$
E_\mu(\alpha)=\left\{x\in \R:\; d_\mu(x)=\alpha\right\},
$$
which is called the {\it level set } of $\mu$.

One of the main objectives of multifractal analysis is to study the {\it dimension spectrum} $\dim_HE_\mu(\alpha)$ and its relation with the $L^q$ spectrum  $\tau_\mu(q)$, here
$\dim_H$ denotes the Hausdorff dimension. The celebrated heuristic principle known as the {\it multifractal
formalism} which was first introduced by some
physicists \cite{H86}, states that for ``good'' measures $\mu$, the dimension spectrum $\dim_HE_\mu(\alpha )$ can be recovered by the  $L^{q}$-spectrum  $\tau_\mu(q)$ through the Legendre transform:
\begin{equation}
\label{e-1.5}
\dim_HE_\mu(\alpha )=\tau^*_\mu(\alpha):=\inf \{\alpha q-\tau_\mu(q):\ q\in \R\}.
\end{equation}
For  more backgrounds  of the multifractal formalism, we refer to the books \cite{Fal-book,Pes97}.
The multifractal  formalism  has been verified to hold for many natural measures including for example,
self-similar measures  satisfying the well-known {\it open set condition} \cite{CaMa92, Ols95, Pat97}. In the recent decade, there have been a lot of interest in studying the validity of the multifractal  formalism for self-similar measures with {\it overlaps} (see, e.g., \cite{FeLa09} and the references therein).

The main  result of the paper is the following.
\begin{thm}
\label{thm-1.1} Let $\lambda\in (1/2,1)$ so that $\lambda^{-1}$ is  a Salem number. Then
\begin{itemize}
\item[(i)]
$E_{\ber}(\alpha)\neq \emptyset$ if $\alpha\in [\tau'_{\nu_\lambda}(+\infty), \tau'_{\ber}(0+)]$, where
$
\tau'_{\nu_\lambda}(+\infty):=\lim_{q\to +\infty}\tau_{\ber}(q)/{q},
$ and $\tau'_{\ber}(0+)$ denotes the right derivative of $\tau_{\ber}$ at $0$.
\item[(ii)]
For any $\alpha\in [\tau'_{\nu_\lambda}(+\infty),\tau'_{\ber}(0+)]$,
 \begin{equation}
 \label{e-k}
 \dim_HE_{\ber}(\alpha)=\tau^*_{\ber}(\alpha):=\inf\{\alpha q-\tau_{\ber}(q):\; q\in \R\}.
 \end{equation}
 \end{itemize}
\end{thm}

In short, the above theorem says that the Bernoulli convolution $\nu_\lambda$ fulfils the multifractal formalism over $q>0$, when $\lambda^{-1}$ is  a Salem number.   As an application, we obtain the following information about the range of local dimensions of $\ber$ associated with certain  Salem numbers.

\begin{thm}
\label{thm-1.2}
For $n\geq 4$, let  $\beta_n$ be the largest real root of the polynomials
$x^n-x^{n-1}-\cdots-x+1$, and let $\lambda_n=\beta_n^{-1}$. Then for $\lambda=\lambda_n$, $\tau'_{\nu_{\lambda}}(+\infty)<1\leq \tau'_{{\ber}}(0+)$; and hence
the range of local dimensions of $\ber$ contains a non-degenerate interval.
\end{thm}

The above results shed somewhat new light on the study of Bernoulli convolutions.  In \cite{Sol07} Solomyak asked whether the multifractal analysis can provide some information about the range of local dimensions of Bernoulli convolutions associated with non-Pisot numbers. Theorem \ref{thm-1.2} provides a positive answer.

Theorem \ref{thm-1.2} also provides a hint  that  $\nu_{\lambda_n}$ might be singular for all $n\geq 4$. Nevertheless, this hint is not direct, since there exists a self-similar measure $\mu$ on $\R$ such that $\mu$ is absolutely continuous and the range of local dimensions of $\mu$ contains a non-degenerate interval  on which the multifractal formalism is valid (see Proposition \ref{pro-5}). This unexpected phenomena looks quite interesting.

Let us give some historic remarks. In the literature there have been a lot of works considering the multifractal structure of Bernoulli convolutions associated with Pisot numbers (see, e.g., \cite{LePo96, Hu97, LaNg98, Por98, Lal98, LaNg99, LaNg99a, Fen03, FeOl03, Fen05, Fen09, FeLa09}). Here we give a brief summary. Assume that $\lambda^{-1}$ is a Pisot number in $(1,2)$. In this case,  the local distribution of $\ber$ can be characterized via matrix products, and as a result, the local dimensions of $\ber$ can be described as the Lyapunov exponents of the associated random matrices, whilst the $L^q$-spectrum  corresponds to the pressure function of  matrix products \cite{Lal98, Fen05, Fen03}. It was shown by    Lau and Ngai \cite{LaNg99} that   $\nu_\lambda$ satisfies the weak separation condition, and \eqref{e-k} holds for those $\alpha=\tau'_{\nu_\lambda}(q)$, $q>0$, provided that $\tau'_{\ber}(q)$ exists. Later in \cite{Fen03} we proved that, indeed,  $\tau_{\ber}$ is  differentiable on $(0, +\infty)$. Recently in  \cite{Fen09}, it was shown that  there exists an interval $I$ in the support of $\ber$ so that,  for the restriction of $\nu_\lambda$ on $I$,  the multifractal formalism is valid on the whole range of the local dimensions, regardless of whether there are phase transitions at $q<0$. This result is extended to self-similar measures satisfying the weak separation condition \cite{FeLa09}.
The $L^q$ spectra and the dimension spectra can be computed explicitly in some concrete cases.
For  $\lambda=\frac{\sqrt{5}-1}{2}$ (the golden ratio case), an explicit formula of $\tau_{\ber}(q)$ on $q>0$ was obtained in \cite{LaNg99a} and was extended to $q\in \R$ in \cite{Fen05}; it was shown in \cite{Fen05} that  $\tau_{\ber}$ has a non-differentiable point in $(-\infty,0)$ (the so-called {\it phase transition} behavior); nevertheless, \eqref{e-k} still holds for all those $\alpha\in [\tau'_{\ber}(+\infty), \tau'_{\ber}(-\infty)]$ \cite{FeOl03}.
The phase transition behaviors and exceptional multifractal phenomena were further found and considered in \cite{LaWa05,  Shm05, Tes06a} for other self-similar measures.
Rather than the golden ratio case, the explicit formulas of  the $L^q$ spectra and the dimension spectra of $\nu_\lambda$ were obtained in \cite{Fen05, OST05}  when $\lambda$ is the unique positive root of $x^n+x^{n-1}+\cdots+x-1$, $n\geq 3$;  in this case,   $\tau_{\ber}$ is  differentiable over $\R$.

When $\lambda$ is an arbitrary number in $(1/2,1)$, the only known result so far  is  that $E_{\ber}(\alpha)\neq \emptyset$ and \eqref{e-k} holds for those $\alpha=\tau'_{\ber}(q)$,  $q>1$, provided that $\tau'_{\ber}(q)$ exists at $q$; and this result extends to all self-conformal measures \cite{Fen07} \footnote{This result also holds for almost all projections of self-conformal measures  \cite{BaBH00}.}. In the case that $\lambda^{-1}$ is a Salem number, the condition $q>1$ can be relaxed to $q>0$  \cite{Fen07}.   However, it still remains open  whether $\tau_{\ber}$ is differentiable over $(0,\infty)$ for each $\lambda$. Although by  concavity $\tau_{\ber}$ has at most countably many non-differentiable points, no much information can be provided for  the range $\{\alpha:\; \alpha=\tau_{\ber}^\prime(q) \mbox{ for some } q>0\}$.

Let us illustrate the main idea in our  proof of Theorem \ref{thm-1.1}. Assume that  $\lambda^{-1}$ is a Salem number in $(1,2)$. The IFS $\{\lambda x-1,\lambda x+1\}$ may not satisfy the weak separation condition (see Remark \ref{rem-LN}), hence the previous approaches via matrix products and the thermodynamic formalism in \cite{Fen09, FeLa09}  are not efficient in this new setting.     For $n\in \N$, denote
$$
t_n=\sup_{x\in \R}\#\{S_{i_1\ldots i_n}:\; i_1\ldots i_n\in \{1,2\}^n, \; S_{i_1\ldots i_n}(K)\cap [x-\lambda^n, x+\lambda^n]\neq \emptyset  \},
$$
where $S_1, S_2$ are given as in \eqref{1.1}, $S_{i_1\ldots i_n}:=S_{i_1}\circ \cdots \circ S_{i_n}$ and $K:=[-\frac{1}{1-\lambda},\frac{1}{1-\lambda}]$ is the attractor of $\{S_1, S_2\}$.  The following simple property is our starting point (see, e.g. \cite{Fen07} for a proof):
\begin{equation}
\label{e-kk}
\lim_{n\to \infty}\frac{\log t_n}{n}=0.
\end{equation}
 Due to  this property, we can manage to setup the following local box-counting principle.
Let  $n\in \N$, $x\in \R$ with $\ber(B_{2^{-n-1}}(x))>0$. Let  $q>0$ so that $\alpha=\tau_{\ber}'(q)$ exists and let $k\in \N$.
Then when $m$ is suitably large (which can be controlled delicately by $n, q, k$ and $\ber(B_{2^{-n}}(x))/\ber(B_{2^{-n-1}}(x))$), there exist $N\geq 2^{m(\tau_{\ber}^*(\alpha)-1/k)}$ many disjoint balls $B_{2^{-n-m}}(x_i)$, $i=1,\ldots, N$, contained in $B_{2^{-n}}(x)$ such that
$$
\frac{\ber(B_{2^{-n-m}}(x_i))}{\ber(B_{2^{-n}}(x))}\in  \left(2^{-m(\alpha+1/k)},\;2^{-m(\alpha-1/k)}\right),
$$
and $\ber(B_{2^{-n-m+1}}(x_i))/\ber(B_{2^{-n-m-1}}(x_i))$ is bounded from above by a constant independent of  $n,m$.  This local box-counting principle is much stronger than  the standard box-counting principle originated in \cite{H86} (see also, Proposition 3.3 in \cite{FeLa09}). According to this principle,
for any $\alpha\in [\tau'_{\nu_\lambda}(+\infty), \tau'_{\ber}(0+)]$, we can give a delicate construction of a Cantor-type subset of
$E_{\ber}(\alpha)$ with Moran structure such that its Hausdorff dimension is greater or equal to  $\tau_{\ber}^*(\alpha)$; this shows that $\dim_HE_{\ber}(\alpha)=\tau_{\ber}^*(\alpha)$, since the upper bound $\dim_HE_{\ber}(\alpha)\leq \tau_{\ber}^*(\alpha)$ always holds (see, e.g., Theorem 4.1 in \cite{LaNg99}).

 Using the similar idea, we can extend  the result of Theorem \ref{thm-1.1}  to  any self-conformal measure which  satisfies  the asymptotically weak separation condition (see Def.~\ref{de-3.1}). That is,
 \begin{thm}
\label{thm-1.3} Let $\nu$ be a self-conformal measure on $\R^d$ satisfying the asymptotically weak separation condition.  Then
for $\alpha\in [\tau'_{\nu}(+\infty), \tau'_{\nu}(0+)]$,
$E_{\nu}(\alpha)\neq \emptyset$ and $\dim_HE_{\nu}(\alpha)=\tau^*_{\nu}(\alpha)$.
 \end{thm}

We remark that the asymptotically weak separation condition is strictly weaker than the weak separation condition introduced in \cite{LaNg99} (see Remark \ref{rem-LN}).

Shortly after the first version of this paper was completed, Jordan, Shmerkin and Solomyak \cite{JSS11} obtained an interesting related result:  for every $\lambda\in (1/2, \gamma)$ where $\gamma\approx 0.554958$ is the root of $1 = x^{-1} +\sum_{n=1}^\infty x^{-2n}$, and $p\in (0, 1/2)$, the biased Bernoulli convolution $\nu_\lambda^p$ (which is the the infinite convolution product of the distributions $p\delta_{-\lambda^n}+(1-p)\delta _{\lambda^n}$)  always contains a non-trivial interval in the range of its local dimensions. It is unknown  whether or not the multifractal formalism holds   for $\nu_\lambda^p$ on this interval.

The paper is arranged in the following manner: in Sect.~\ref{S-2}, we show that for a general measure $\mu$ in $\R^d$, the multifractal formalism is valid if certain local box-counting principle holds for $\mu$; we prove Theorem \ref{thm-1.3} in Sect.~\ref{S-3} by showing that  this local box-counting principle holds for self-conformal measures on $\R^d$ satisfying the asymptotically weak separation condition; in Sect.~\ref{S-4}, we prove Theorem \ref{thm-1.2}; in Sect.~\ref{S-5}, we construct an example of absolutely continuous self-similar measure on $\R$ with non-trivial range of local dimensions.

\section{A general scheme for the validity of  the multifractal formalism}
\label{S-2}
Let $\mu$ be  a finite Borel measure $\mu$ in $\R^d$ with compact support.
 Let $\tau(q):=\tau_\mu(q)$ be the $L^q$-spectrum of $\mu$, and let $E(\alpha):=E_\mu(\alpha)$ denote the level set of $\mu$.  (See Sect.\ref{S-1} for the definitions.) Assume that $\tau(q)\in \R \mbox{ for each }q\in \R$. In this section we show that the multifractal formalism is  valid for $\mu$ if  certain local box-counting principle holds for $\mu$.

 Define
\begin{equation}
\label{e-key}
\Omega=\{q\in \R:\;  \mbox{ the derivative $\tau'(q)$ exists}\}\quad \mbox{and}\quad  \Omega_+=\Omega\cap (0,\infty).
\end{equation}
Since $\tau$ is concave on $\R$, $\Omega$ is dense in $\R$ and $\Omega_+$ is dense in $(0,\infty)$.
\begin{de}
\label{de-2.1}
{\rm
We say that $\mu$ has an  {\it asymptotically good  multifractal  structure over $\R$} (resp., $\R_+$) if there is a dense subset  $\Lambda$ of $\Omega$ (resp. $\Omega_+$) such that for each $q\in \Lambda$ and $k\in \N$,
there exist positive numbers $a(q,k)$, $b(q,k)$, $f_n(q,k)$, $n=0, 1, 2,\cdots$,  such that the following properties hold:
\begin{itemize}
\item[(i)]
\begin{equation}\label{e-1.0}
 \lim_{k\to\infty} b(q,k)=0,\qquad \lim_{n\to \infty}f_n(q,k)/n=0.
\end{equation}
\item[(ii)]Let $n\geq 0$ and $x\in \R$ so that  $\mu(B_{2^{-n-1}}(x))>0$. Then for any integer $m$ with
\begin{equation}
\label{e-1.2}
m\geq f_n(q,k)+ a(q,k) \log \frac{\mu(B_{2^{-n}}(x))}{\mu(B_{2^{-n-1}}(x))},
\end{equation}
there are disjoint balls $B_{2^{-n-m}}(x_i)\subset B_{2^{-n}}(x)$, $i=1,\ldots,N$, such that
\begin{equation*}
\label{e-1.3}
N\geq  2^{m(\tau^\prime(q) q-\tau(q)-b(q,k))},
\end{equation*}
\begin{equation*}
\label{e-1.4}
2^{-m(\tau'(q)+1/k)}\leq \frac{\mu(B_{2^{-n-m}}(x_i))}{\mu(B_{2^{-n}}(x))}\leq 2^{-m(\tau'(q)-1/k)},
\end{equation*}
and
\begin{equation*}
\label{e-1.4'}
\frac{\mu(B_{2^{-n-m+1}}(x_i))}{\mu(B_{2^{-n-m-1}}(x_i))}\leq f_{n+m}(q,k).
\end{equation*}
\end{itemize}
}
\end{de}
\bigskip

The main result in this section is the following.

\begin{thm}
\label{thm-2.1}
 {\rm (a)} Assume that $\mu$ has an asymptotically good multifractal structure over $\R$. Let $\alpha_{\min}=\lim_{q\to \infty}\tau(q)/q$ and $\alpha_{\max}=\lim_{q\to -\infty} \tau(q)/q$. Then
$E(\alpha)\neq \emptyset$ if and only if  $\alpha\in [\alpha_{\min},\, \alpha_{\max}]\cap \R$.\footnote{$\alpha_{\min}$ is always non-negative and finite. It is possible that $\alpha_{\max}=+\infty$.} Furthermore, for any $\alpha\in [\alpha_{\min},\, \alpha_{\max}]\cap \R$, $$\dim_HE(\alpha)=\tau^*(\alpha)=\inf\{\alpha q-\tau(q):\; q\in \R\}.$$
{\rm (b)} Assume that $\mu$ has an asymptotical multifractal structure over $\R^+$. Then for $\alpha\in [\alpha_{\min}, \tau^\prime(0+)]$, we have $E(\alpha)\neq \emptyset$ and $\dim_HE(\alpha)=\tau^*(\alpha)$.
\end{thm}

A key idea in the proof of the above theorem is  to construct Cantor-type subsets of $E(\alpha)$ with a special  Moran construction.

\begin{de}
{\rm
Let $B\subset {\R}^d$ be a closed ball. Let $\{N_\ell\}_{\ell\geq 1}$ be a sequence
of positive integers. Let $D=\bigcup_{\ell\geq 0}D_\ell$ with
$D_0=\{\emptyset\}$ and $D_\ell=\{\omega = (i_1i_2\cdots i_\ell): 1\leq
i_j\leq N_j, 1\leq j\leq \ell\}$.  Suppose that ${\mathcal G}=\{B_\omega:
\omega\in D\}$ is a collection of  closed balls of radius $r_\omega$  in ${\R}^d$. We say that
${\mathcal G}$ fulfills the {\it Moran structure}, provided it satisfies the following conditions:
\begin{itemize}
\item[(1)] $B_\emptyset = B$, $B_{\omega j} \subset
            B_\omega $ for any $\omega \in D_{\ell-1}, 1 \leq j \leq
N_\ell$;
\item[(2)] $B_{\omega}\cap B_{\omega'}= \emptyset$ for $\omega,\omega'\in D_\ell$ with $\omega\neq \omega'$.

\item[(3)] $\lim_{k\rightarrow \infty} \max_{\omega\in D_\ell}
r_\omega = 0$;

\item[(4)] For all $\omega \eta \neq \omega^\prime \eta, \ \omega,
\omega^\prime \in D_m, \ \omega \eta,  \omega^\prime \eta \in D_n, m \leq n$,
\begin{equation*}
 \frac {r_{\omega \eta}}{r_\omega } \
= \frac {r_{\omega^\prime \eta}} {r_{\omega^\prime} }. \
 \end{equation*}
\end{itemize}

\noindent If ${\mathcal G}$ fulfills the above Moran structure, we
call
$$
F=\bigcap_{\ell=1}^\infty\bigcup_{\omega\in D_\ell} B_\omega
$$
the {\it  Moran set} associated with
${\mathcal G}$.
}

\end{de}

For $\ell\in \N$, let
$$
c_\ell=\min_{(i_1\cdots i_\ell)\in D_\ell} \frac{r_{i_1\cdots
i_\ell}}{r_{i_1\cdots i_{\ell-1}}},\ \ \ M_\ell=\max\limits_{(i_1\cdots
i_\ell)\in D_\ell} r_{i_1\cdots i_\ell}.
$$

\medskip

\begin{pro} \label{pro-2.4}
{\cite[Proposition 3.1]{FLW02}}.
For a Moran set $F$ defined as above, suppose furthermore
\begin{equation}\label{e-2.7}
\lim_{k\rightarrow \infty}\frac{\log c_\ell}{\log M_\ell} = 0.
\end{equation}
 Then we have
$$
\dim_HF=\liminf_{\ell\rightarrow \infty}s_\ell,\ \
$$
where $s_\ell$
satisfies the equation $\displaystyle \sum_{\omega \in
D_\ell}r_\omega^{s_\ell}=1$ for each $k$.
\end{pro}

\bigskip

\begin{proof}[Proof of Theorem \ref{thm-2.1}] We only prove part (a) of the theorem, since the proof of part (b) is  essentially identical. We divide the proof into several steps.

{\em Step 1. If $\alpha\in  \overline{\{\tau^\prime(q):\; q\in \Omega\}}$, then $E(\alpha)\neq \emptyset$ and $\dim_HE(\alpha)\geq \tau^*(\alpha)$.}

Let $\Lambda$ and $a(q,k),b(q,k), f_n(q,k)$ ($q\in \Lambda$, $k,n\in \N$) be given as in Def.~\ref{de-2.1}.
We can assume that $\lim_{n\to \infty} f_n(q,k)=\infty$, since in Def.~\ref{de-2.1},  we can change $f_n(q,k)$ to $\max\{f_n(q,k),\;\log n\}$ with no harm.

Fix  $\alpha\in  \overline{\{\tau^\prime(q):\; q\in \Omega\}}$. Since $\tau$ is  concave on $\R$ and $\Lambda$ is dense in $\Omega$, there exists a sequence $(q_j)_{j=1}^\infty\subset \Lambda$
such that $\lim_{j\to \infty}\tau^\prime(q_j)=\alpha$. Note that $\tau^*$ is also  concave (and hence lower semi-continuous)  on $[\alpha_{\min},\alpha_{\max}]\cap \R$ (see \cite{Roc-book}).
Hence
\begin{equation}
\label{e-v3}
\tau^*(\alpha)\leq \liminf_{j\to \infty}\tau^*(\tau^\prime(q_j))=\liminf_{j\to \infty}(\tau^\prime(q_j)q_j-\tau(q_j)).
\end{equation}
 Take a sequence $(k_j)_{j=1}^\infty$ of positive integers such that $\lim_{j\to \infty} k_j=\infty$ and
\begin{equation}
\label{e-v4}
b_j:=b(q_j,k_j)\rightarrow 0, \mbox{ as } j\to \infty.
\end{equation}
Pick $x_0\in \R$ such that  $\mu(B_{1/2}(x_0))>0$. Set
$$A_0=\frac{\mu(B_1(x_0))}{\mu(B_{1/2}(x_0))}.$$
Clearly $1\leq A_0<\infty$.
Then due to \eqref{e-1.0},  we can define a sequence $(L_j)_{j=1}^\infty$ of positive integers recursively such that $L_1\geq 2$ and
\begin{equation}
\label{e-a1}
n\geq f_0(q_1,k_1) +a(q_1,k_1)\log A_0\quad \mbox{ if } n\geq L_1
\end{equation}
 and
\begin{equation}
\label{e-a1'}
n\geq f_n(q_1,k_1) +a(q_1,k_1)\log f_n(q_1, k_1)\quad \mbox{ if } n\geq L_1
\end{equation}
and
\begin{equation}
\label{e-a2}
\frac{n}{j+1}\geq f_n(q_{j+1},k_{j+1})+ a(q_{j+1},k_{j+1})\log (f_n(q_{j+1},k_{j+1})+ f_n(q_{j},k_{j}))
\end{equation}
if $n\geq L_j$, $j=1,2, \ldots.$

  Construct a sequence of positive integers $(n_\ell)_{\ell=1}^\infty$  recursively  by setting  $n_1=L_1$ and for $\ell\geq 2$,
 \begin{equation}
 \label{e-a3}
 n_\ell=\mbox{ the smallest integer greater than } (n_1+\cdots+n_{\ell-1})/\theta(\ell),
    \end{equation}
        where  $\theta(\ell)$  denotes the unique positive integer $j$ satisfying $$L_0+\cdots+L_{j-1}\leq \ell< L_0+\cdots+L_j.$$
 Here we take the convention  $L_0=0$.
 Clearly,
 \begin{equation}
 \label{e-v1}
 0\leq \theta(\ell+1)-\theta(\ell)\leq 1,\quad \lim_{\ell\to \infty}\theta(\ell)=\infty,\quad\lim_{\ell\to \infty}\frac{\theta(\ell+1)}{
 \theta(\ell)}=1.
 \end{equation}
 Moreover,
  \begin{equation}
 \label{e-v2}
 \lim_{\ell\to \infty}\frac{n_1+\cdots+n_{\ell-1}}{n_1+\cdots+ n_\ell}=\lim_{\ell\to \infty}\frac{n_1+\cdots+n_{\ell-1}}{(n_1+\cdots+ n_{\ell-1})(1+1/\theta(\ell))}=1.
 \end{equation}
 Combining \eqref{e-v1}, \eqref{e-v2} and \eqref{e-a3}, we have
    \begin{equation}
    \label{e-a4}
     \lim_{\ell\to \infty}\frac{n_\ell}{n_1+\cdots+n_{\ell-1}}=0,\qquad \lim_{\ell\to \infty}\frac{n_{\ell}}{n_{\ell-1}}=\lim_{\ell\to \infty}\frac{(n_1+\cdots+n_{\ell-1})/\theta(\ell)}{(n_1+\cdots+ n_{\ell-2})/\theta(\ell-1)}=1.
    \end{equation}

By \eqref{e-a1}, we have
\begin{equation}\label{e-w1}
n_1=L_1\geq f_0(q_1,k_1)+ a(q_1,k_1)\log \frac{\mu(B_1(x_0))}{\mu(B_{1/2}(x_0))}.
\end{equation}
We claim that for any $\ell\geq 1$,
\begin{equation}\label{e-w2}
n_{\ell+1}\geq f_{n_1+\cdots +n_{\ell}}(q_{\theta(\ell+1)}+ k_{\theta(\ell+1)})+ a(q_{\theta(\ell+1)},k_{\theta(\ell+1)})\log f_{n_1+\cdots+ n_\ell} (q_{\theta(\ell)}, k_{\theta(\ell)}).
\end{equation}
To prove \eqref{e-w2}, fix $\ell$ and set $j=\theta(\ell+1)$.  First we consider the case that $j=1$. In this case, by \eqref{e-a3},
$n_{\ell+1}\geq n_{1}+\cdots+n_{\ell}$. Note that in this case $\theta(\ell)=1$,  hence \eqref{e-w2} follows from \eqref{e-a1'}.
Next we assume $j\geq 2$. Then $\theta(\ell)=j$ or $j-1$. By the definition of $\theta$, $$L_{j-1}\leq \ell+1\leq n_1+\cdots +n_{\ell}.$$
Since $n_{\ell+1}\geq (n_1+\cdots +n_{\ell})/j$, \eqref{e-w2} follows from \eqref{e-a2}.

Denote $\lambda_j=\tau^\prime(q_j)q_j-\tau(q_j)-b_j$ for $j\in \N$. Then by \eqref{e-v3}-\eqref{e-v4}, we have
\begin{equation}
\label{e-w10}
\liminf_{j\to \infty} \lambda_j\geq \tau^*(\alpha).
\end{equation}
Define a sequence $(N_\ell)_{\ell=1}^\infty$ by
$$
N_\ell=\max\left\{1,\; \left[2^{n_\ell\lambda_{\theta(\ell)}}\right]\right\},
$$
where $[x]$ denotes the integer part of $x$.

 Let $D=\bigcup_{\ell\geq 0}D_\ell$ with
$D_0=\{\emptyset\}$ and $D_\ell=\{\omega = (i_1i_2\cdots i_\ell): 1\leq
i_j\leq N_j, 1\leq j\leq \ell\}$.  We will construct a collection ${\mathcal G}=\{B_\omega:
\omega\in D\}$ of  closed balls of radius $r_\omega$  in ${\R}^d$ recursively, which has Moran structure and satisfies the following properties:
\begin{itemize}
\item[(p1)] $B_{\emptyset}=B_1(x_0)$;
\item[(p2)] $r_\omega=2^{-(n_1+\cdots+n_{\ell})}$ for each $\omega\in D_\ell$;
\item[(p3)] For each $\ell\geq 1$, $\omega\in D_{\ell-1}$ and $1\leq i\leq N_{\ell}$,
\begin{equation*}
2^{-n_{\ell}(\tau^\prime(q_{\theta(\ell)})+1/k_{\theta(\ell)})}\leq \frac{\mu(B_{\omega i})}{\mu(B_\omega)}\leq 2^{-n_{\ell}(\tau^\prime(q_{\theta(\ell)})-1/k_{\theta(\ell)})}
\end{equation*}
and
$$\mu(2B_{\omega i})/\mu(\frac12 B_{\omega i})\leq f_{n_1+\cdots+n_{\ell}}(q_{\theta(\ell)}, k_{\theta(\ell)})\leq n_1+\cdots+n_{\ell},$$
here and afterwards,  $cB$ denotes $B_{cr}(x)$ when $B=B_{r}(x)$.
\end{itemize}

The construction is done by induction. We first set $B_\emptyset=B_1(x_0)$. Since $\mu$ has an asymptotical multifractal structure, by  \eqref{e-w1} and Def.~\ref{de-2.1},
there exist $N_1$ disjoint closed balls $\{B_i\}_{i=1}^{N_1}$ of radius $2^{-n_1}$, contained in $B_\emptyset$, such that
$$
2^{-n_1(\tau^\prime(q_1)+{1}/{k_{1}})}\leq \frac{\mu(B_i)}{\mu(B_\emptyset)}\leq 2^{-n_1(\tau^\prime(q_1)-1/k_1)}
$$
and $$\frac{\mu(2B_i)}{\mu(\frac12 B_i)}\leq f_{n_1}(q_1, k_1)\leq n_1.$$
   Relabel this family of  $N_1$ balls by $\{B_\omega: \omega\in D_1\}$. Then (p3) holds in the case $\ell=1$ (noting that $\theta(1)=1$).

Assume we have constructed well the family of disjoint balls $\{B_\omega: \omega\in D_\ell\}$ for some $\ell\geq 1$ so that
each ball in this family has radius $2^{-n_1-\cdots-n_\ell}$, and (p3) holds for $\ell$.   Next we construct $\{B_{\omega'}: \omega'\in D_{\ell+1}\}$.
Fix $\omega\in D_\ell$.  Since (p3) holds for $\ell$, we have
$$
\mu(B_\omega)/\mu\big(\frac{1}{2}B_\omega\big)\leq f_{n_1+\cdots+n_\ell}(q_{\theta(\ell)}, k_{\theta(\ell)}).
$$
Combining the above inequality with \eqref{e-w2} yields
$$
n_{\ell+1}\geq f_{n_1+\cdots +n_{\ell}}(q_{\theta(\ell+1)}, k_{\theta(\ell+1)})+ a(q_{\theta(\ell+1)},k_{\theta(\ell+1)})\log \frac{\mu(B_\omega)}{\mu\big(\frac{1}{2}B_\omega\big)}.
$$
By Def.~\ref{de-2.1}, there exist $N_{\ell+1}$ disjoint balls of radius $2^{-n_1-\cdots-n_{\ell+1}}$, which we denote as
$B_{\omega i}$, $i=1,\ldots, N_{\ell+1}$, such that $B_{\omega i}\subset B_\omega$ and
$$
2^{-n_{\ell+1}(\tau^\prime(q_{\theta(\ell+1)})+1/k_{\theta(\ell+1)})}\leq \frac{\mu(B_{\omega i})}{\mu(B_\omega)}\leq 2^{-n_{\ell+1}(\tau^\prime(q_{\theta(\ell+1)})-1/k_{\theta(\ell+1)})}
$$
and $$\frac{\mu(2B_{\omega i})}{\mu(\frac12 B_{\omega i})}\leq f_{n_1+\cdots +n_{\ell+1}} (q_{\theta(\ell+1)}, k_{\theta(\ell+1)}).$$
Now letting $\omega$ vary over $D_\ell$, we get the family $\{B_{\omega i}:\; \omega\in D_\ell,\; 1\leq i\leq N_{\ell+1}\}:= \{B_{\omega'}: \omega'\in D_{\ell+1}\}$.
Clearly,  (p3) holds for $\ell+1$.

Hence by induction, we can construct well ${\mathcal G}:=\{B_\omega:\; \omega\in D\}$ which has the Moran structure and satisfies (p1)-(p3). Clearly, by (p3),  for each $\ell\geq 1$ and $\omega\in D_\ell$ we have
\begin{equation}
\label{e-w11}
\prod_{i=1}^\ell 2^{-n_i (\tau'(q_{\theta(i)})+1/k_{\theta(i)})} \leq \frac{\mu(B_\omega)}{\mu(B_\emptyset)}\leq \prod_{i=1}^\ell 2^{-n_i (\tau'(q_{\theta(i)})-1/k_{\theta(i)})}.
\end{equation}

Let $F=\bigcap_{\ell=1}^\infty\bigcup_{\omega\in D_\ell} B_\omega$ be the Moran set associated with ${\mathcal G}$. We can use Proposition \ref{pro-2.4} to determine the Hausdorff dimension of $F$. Indeed in our case, $c_\ell=2^{-n_\ell}$ and $M_\ell=2^{-n_1-\cdots-n_\ell}$, hence by \eqref{e-a4}, the assumption \eqref{e-2.7} fulfills.
Thus by Proposition \ref{pro-2.4} and \eqref{e-w10},
$$
\dim_HF=\liminf_{\ell\to \infty} \frac{\log (N_1\cdots N_\ell)}{\log (2^{n_1+\cdots+n_\ell})}\geq \liminf_{\ell\to \infty} \lambda_{\theta(\ell)}\geq \tau^*(\alpha).
$$

In the end of this step, we show that $F\subset E(\alpha)$ and hence  $\dim_HE(\alpha)\geq \dim_HF\geq  \tau^*(\alpha)$.  To see this, let $x\in F$. Let $r>0$ be a small number. Then there exists $\ell\geq 1$ such that
\begin{equation}\label{e-w13}
 2^{-n_1-\cdots-n_{\ell+1}}\leq r<2^{-n_1-\cdots-n_\ell}.
 \end{equation}
  Clearly, $B_r(x)$ contains a ball, say $B_{\omega'}$, for some $\omega'\in D_{\ell+2}$.  On the other hand,
$B_r(x)$ intersects at least one ball, say $B_\omega$, for some $\omega\in D_\ell$, which implies $B_r(x)\subseteq 2B_\omega$. Hence we have
\begin{equation}
\label{e-w12}
\mu(B_r(x))\geq \mu(B_{\omega'})\quad \mbox{and}\quad \mu(B_r(x))\leq \mu(2B_\omega)\leq (n_1+\cdots+n_\ell) \mu(B_\omega).
\end{equation}
Combining \eqref{e-w12} with \eqref{e-w11}, \eqref{e-w13} and  \eqref{e-a4} yields
$$
\lim_{r\to 0}\frac{\log \mu(B_r(x))}{\log r}=\lim_{i\to \infty}\tau'(q_{\theta(i)})+1/k_{\theta(i)}=\alpha.
$$
That is, $x\in E(\alpha)$. Hence we have $F\subset E(\alpha)$. This finishes the proof of Step 1.

\bigskip
{\em Step 2. If $\alpha=p\tau^\prime(q_1)+(1-p)\tau^\prime(q_2)$ for some $0<p<1$ and $q_1,q_2\in \Omega$,
then $E(\alpha)\neq \emptyset$ and $\dim_HE(\alpha)\geq p\tau^*(\alpha_1)+(1-p)\tau^*(\alpha_2)$, where
$\alpha_1:=\tau^\prime(q_1)$ and $\alpha_2=\tau^\prime(q_2)$.}

The proof of this step is quite similar to that in Step 1. We only list the main different points.

Fix $q_1,q_2\in \Omega$ and $0<p<1$. Since $\Lambda$ is dense in $\Omega$, there exist two sequences $(q_{1,j})_{j=1}^\infty$, $(q_{2,j})_{j=1}^\infty\subset \Lambda$
such that $\lim_{j\to \infty} q_{i,j}=q_i$, $i=1,2$. Since $\tau$ is concave, we have $\lim_{j\to \infty}\tau'(q_{i,j})=\tau'(q_i)=\alpha_i$, $i=1,2$.   By \eqref{e-1.0}, there exists a sequence of integers $(k_j)\uparrow \infty$ such that
$\lim_{j\to \infty} b(q_{i,j}, k_j)=0$.

By \eqref{e-1.0}, we can define a sequence $(L_j)_{j=0}^\infty$  of integers such that   $L_0=0$ and for $j\geq 1$,
$$
n\geq f_0(q_{i, 1},k_1) +a(q_{i,1},k_1)\log A_0\quad \mbox{ if } n\geq L_1, \; i=1,2,
$$
$$
n\geq f_n(q_{i,1},k_1) +a(q_{i, 1},k_1)\log f_n(q_{i,1}, k_1)\quad \mbox{ if } n\geq L_1, \; i=1,2,
$$
and
$$\frac{n}{j+1}\geq f_n(q_{i, j+1},k_{j+1})+ a(q_{i, j+1},k_{j+1})\log (f_n(q_{i, j+1},k_{j+1})+ f_n(q_{i, j},k_{j}))
$$
if $n\geq L_j$, $j=1,2, \ldots$, $i=1,2$.  Note that the sequence $(L_j)_{j=0}^\infty$ may be different from what we constructed in Step 1.

Construct $(n_\ell)_{\ell=1}^\infty$ from $(L_j)_{j=0}^\infty$ in the same way as in Step 1. Again, we use $\theta(\ell)$ denote the unique positive integer $j$ satisfying $\sum_{s=0}^{j-1}L_s\leq \ell<\sum_{s=0}^{j}L_s.$

For $\ell\geq 1$, set
\begin{equation}
\label{e-z1}
t_\ell=\left\{
\begin{array}{ll}
1 & \mbox{  if } \{\ell\sqrt{2}\}\in [0,p),\\
2 & \mbox{  if } \{\ell\sqrt{2}\}\in [p,1),\\
\end{array}
\right.
\end{equation}
where $\{x\}$ denotes the fractional part of $x$,  and define
$$
u_\ell=\tau^\prime(q_{t_\ell, \theta(\ell)})q_{t_\ell, \theta(\ell)}-\tau(q_{t_\ell, \theta(\ell)})-b(q_{t_\ell, \theta(\ell)}, k_{\theta(\ell)}).
$$
It is easy to check that
\begin{equation}
\label{e-z2}
\lim_{\ell\to \infty} (u_\ell-\tau^*(\alpha_{t_\ell}))=0.
\end{equation}
Then define a sequence $(N_\ell)_{\ell=1}^\infty$ by
\begin{equation*}
N_\ell=\max\left\{1,\; \left[2^{n_\ell u_\ell}\right]\right\},
\end{equation*}
here $[x]$ denotes the integer part of $x$.

Pick $x_0\in \R$ such that  $\mu(B_{1/2}(x_0))>0$. Let $D=\bigcup_{\ell\geq 0}D_\ell$ with
$D_0=\{\emptyset\}$ and $D_\ell=\{\omega = (i_1i_2\cdots i_\ell): 1\leq
i_j\leq N_j, 1\leq j\leq \ell\}$.  Similar to Step 1, we can  construct a collection ${\mathcal G}=\{B_\omega:
\omega\in D\}$ of  closed balls of radius $r_\omega$  in ${\R}^d$ recursively, which has Moran structure and satisfies the following properties:
\begin{itemize}
\item[(q1)] $B_{\emptyset}=B_1(x_0)$;
\item[(q2)] $r_\omega=2^{-(n_1+\cdots+n_{\ell})}$ for each $\omega\in D_\ell$;
\item[(q3)] For each $\ell\geq 1$, $\omega\in D_{\ell-1}$ and $1\leq i\leq N_{\ell}$,
$$
2^{-n_{\ell}(\tau^\prime(q_{t_\ell, \theta(\ell)})+1/k_{\theta(\ell)})}\leq \frac{\mu(B_{\omega i})}{\mu(B_\omega)}\leq 2^{-n_{\ell}(\tau^\prime(q_{t_\ell, \theta(\ell)})-1/k_{\theta(\ell)})}.
$$
and
$$\mu(2B_{\omega i})/\mu(\frac12 B_{\omega i})\leq f_{n_1+\cdots+n_{\ell}}(q_{t_\ell, \theta(\ell)}, k_{\theta(\ell)})\leq n_1+\cdots+n_{\ell}.$$
\end{itemize}

Let $F=\bigcap_{\ell=1}^\infty\bigcup_{\omega\in D_\ell} B_\omega$ be the Moran set associated with ${\mathcal G}$. Similar to Step 1, we can show that $F\subset E(\alpha)$ and $$\dim_HE(\alpha)\geq \dim_HF=\liminf_{\ell\to \infty} \frac{\log (N_1\cdots N_\ell)}{\log (2^{n_1+\cdots+n_\ell})}\geq p\tau^*(\alpha_1)+(1-p)\tau^*(\alpha_2).$$
This finishes the proof of Step 2.
\bigskip

{\em Step 3. $E(\alpha)\neq \emptyset$ if and only if  $\alpha\in [\alpha_{\min},\, \alpha_{\max}]\cap \R$. Furthermore, for any $\alpha\in [\alpha_{\min},\, \alpha_{\max}]\cap \R$, $\dim_HE(\alpha)=\tau^*(\alpha)=\inf\{\alpha q-\tau(q):\; q\in \R\}.$}

First we show that $E(\alpha)\neq \emptyset$ implies that $\alpha\in [\alpha_{\min},\, \alpha_{\max}]$. Indeed, assume that $\alpha=\lim_{r\to 0}\frac{\log \mu( B_r(x))}{\log r}$ for some $x\in \R$. Then $\Theta(q,r)\geq \mu(B_r(x))^q$ (cf. \eqref{e-tz1}),  which implies $\tau(q)\leq \alpha q$. Hence $\alpha\in [\alpha_{\min}, \alpha_{\max}]$.

Next we show that if $\alpha\in [\alpha_{\min}, \alpha_{\max}]\cap \R$, then $E(\alpha)\neq \emptyset$ and $\dim_HE(\alpha)\geq \tau^*(\alpha)$.
To see this, let $\alpha\in [\alpha_{\min}, \alpha_{\max}]\cap \R$. Since $\tau$ is concave, there are only two possible cases:  (1) $\alpha \in  \overline{\{\tau^\prime(q):\; q\in \Omega\}}$;   (2) $\alpha\in (\tau^\prime(q+), \tau^\prime(q-))$ for some $q\in \R$, here  $\tau^\prime(q+), \tau^\prime(q-)$ denote the right and left derivatives of $\tau$ at $q$, respectively. By Step 1, we only need to consider the second case.  Clearly, there exists $0<p<1$ such that
$$
\alpha=p\tau^\prime(q+)+(1-p) \tau^\prime(q-).
$$
Since $\tau$ is concave, there exist two sequences $(q_j)_{j=1}^\infty$, $(q'_j)_{j=1}^\infty\subset \Omega$ such that
$$
q_j\searrow q,\quad q'_j\nearrow q,\quad \tau'(q_j)\nearrow \tau'(q+), \quad \tau'(q'_j)\searrow \tau'(q-)
$$
as $j$ tends to infinity. Therefore, there exists a sequence $(p_j)_{j=1}^\infty\subset (0,1)$ such that $\lim_{j\to \infty}p_j=p$ and
$$
   \alpha=p_j\tau^\prime(q_j)+(1-p_j) \tau^\prime(q'_j).
$$
By Step 2, we have $E(\alpha)\neq \emptyset$ and
$$
\dim_HE(\alpha)\geq p_j (\tau^\prime(q_j)q_j-\tau(q_j))+(1-p_j) (\tau^\prime(q'_j) q'_j-\tau(q'_j)),\quad j\in \N.
 $$
Letting $j\to \infty$, we obtain
$$\dim_HE(\alpha)\geq (p \tau^\prime(q+)+(1-p) \tau^\prime(q-)) q-\tau(q)=\alpha q-\tau(q)=\tau^*(\alpha).$$

In the end, we point out that if $\alpha\in [\alpha_{\min}, \alpha_{\max}]\cap \R$, then $\dim_HE(\alpha)=\tau^*(\alpha)$. This follows from the basic fact that
$\dim_HE(\alpha)\leq \tau^*(\alpha)$ whenever $E(\alpha)\neq \emptyset$ (indeed,  this fact holds for any compactly supported probability measure; see, e.g., Theorem 4.1 in \cite{LaNg99}). This finishes the proof of Theorem \ref{thm-2.1}.
\end{proof}

\section{Self-conformal measures with the AWSC}
\label{S-3}

In this section we prove Theorem \ref{thm-1.3}. In Sect.~\ref{Sub-1}, we introduce some notation and definitions about self-conformal measures and the asymptotically weak separation condition. In Sect.~\ref{Sub-2}, we show that any self-conformal measure with the asymptotically weak separation condition has an asymptotically multifractal structure on $\R^+$; then Theorem \ref{thm-1.3} follows from Theorem \ref{thm-2.1}(b).

\subsection{Self-conformal measures and asymptotically weak separation condition}\label{Sub-1}
Let $U\subset \R^d$ be an open set. A $C^1$-map $S: U\to \R^d$ is {\it conformal} if the differential $S^\prime(x): \R^d\to \R^d$
satisfies $\left|S^\prime(x)y\right|=\left|S^\prime(x)\right|\cdot |y|\neq 0$
for all $x\in U$ and $y\in \R^d$, $y\neq 0$. Furthermore,  $S: U\to \R^d$ is {\it contracting} if
there exists $0<c<1$ such that $|S(x)-S(y)|\leq c\cdot |x-y|$ for all $x,y\in U$.
We say that $\{S_i: X\to
X\}_{i=1}^\ell $ is a {\it $C^1$-conformal iterated function system} ( $C^1$-conformal
IFS) on a compact set $X\subset \R^d$ if each $S_i$ extends to an injective
contracting $C^1$-conformal map $S_i: U\to U$ on an open set $U\supset X$.

Let $\{S_i\}_{i=1}^{\ell}$ be a $C^1$-conformal IFS on a compact set
$X\subset \R^d$.
It is well-known, see \cite{Hut81}, that there is a
unique non-empty compact set $K \subset X$ such that
$K=\bigcup_{i=1}^\ell S_i(K)$. Given a probability vector $(p_1,\ldots
, p_\ell) $, there is a unique Borel probability measure $\nu$
satisfying
\begin{equation}  \label{e-1.2}
\nu=\sum_{i=1}^\ell p_i \nu\circ S_i^{-1}.
\end{equation}
This measure is supported on $K$ and it is called {\it self-conformal}. In
particular, if the maps $S_i$ are all similitudes, then $\nu$ is called {\it %
self-similar}.

Let ${\mathcal A}=\{1,\ldots,\ell\}$. Denote ${\mathcal A}^*=\bigcup_{n\geq 1}{\mathcal A}^n$. For $%
u=u_1\ldots \circ u_k$, we write $S_u=S_{u_1}\circ \cdots S_{u_k}$, $p_u=p_{u_1}\cdots p_{u_k}$ and $K_u=S_u(K)$%
; in particular we let $\tilde{u}$ denote the word obtained by dropping the
last letter of $u$. For $n\in {\Bbb N}$, denote
\begin{equation}  \label{e-a.1}
W_n:=\left\{u\in {\mathcal A}^*:\; \mbox{\rm diam}(K_u)\leq 2^{-n}, \mbox{\rm diam}(K_{%
\tilde{u}})> 2^{-n}\right\}.
\end{equation}

For $n\geq 0$, let
\begin{equation}
\label{e-dn}
{\bf D}_{n}=\{[0, 2^{-n})^d+{\bf v}:\; {\bf v}\in 2^{-n}\Z^d\},
\end{equation}
 and define
$$
\tau_n(q)=\sum_{Q\in {\bf D}_{n}}\nu(Q)^q.
$$

\begin{pro}
\label{pro-3.1}
There is a sequence $(c_n)_{n=1}^\infty$ of positive numbers with $$\lim_{n\to \infty}\frac{1}{n}\log c_n=0,$$
 such that for any $q>0$, $ n,m\in \N$, and all $u\in W_n$,
\begin{equation}
\label{e-3.1}
(c_n)^{-(q+1)}\tau_m(q)\leq \sum_{Q\in {\bf D}_{m+n}}(\nu(S_u^{-1}Q))^q\leq   (c_n)^{q+1}\tau_m(q).
\end{equation}
Furthermore, the limit $\lim_{m\to \infty} \frac{\log \tau_m(q)}{-m\log 2}$ exists for each $q>0$ and it coincides with $\tau(q):=\tau_\nu(q)$ defined as in Sec~\ref{S-1}.
\end{pro}
\begin{proof}
It was proved in \cite[Proposition 3.3]{Fen07} that  there exists $\beta>0$ such that for any $\epsilon>0$, there exists $C(\epsilon)>0$ such that for all  $q>0$,  $m,n\in \N$, and all $u\in W_n$,
\begin{equation}
\label{e-3.2}
\left(C(\epsilon)(1+\epsilon)^{\beta n}\right)^{-(q+1)} \tau_m(q)\leq \sum_{Q\in \D_{m+n}}(\nu(S_u^{-1}Q))^q\leq \left(C(\epsilon)(1+\epsilon)^{\beta n}\right)^{q+1} \tau_m(q).
\end{equation}
Choose a sequence of positive numbers $(\epsilon_n)$ tending to $0$ slowly enough such that $\lim_{n\to \infty} (1/n) \log C(\epsilon_n)=0$.
Let $c_n=C(\epsilon_n) (1+\epsilon_n)^{\beta n}$. Then $\lim_{n\to \infty}(\log c_n)/n=0$, and \eqref{e-3.1} follows from \eqref{e-3.2}.
The existence of  $\lim_{m\to \infty} \frac{\log \tau_m(q)}{-m\log 2}$ for each $q>0$ was proved in \cite[Proposition 4.3]{Fen07}.
It is easy to check that the limit coincides with $\tau_\nu(q)$.
\end{proof}

We remark that Proposition \ref{pro-3.1} was first proved by Peres and Solomyak \cite{PeSo00} under the bounded distortion assumption on $\{S_i\}_{i=1}^\ell$. In that case, the involved $(c_n)$ in  \eqref{e-3.1} can be  replaced by a constant $c$.

The following definition was introduced in \cite{Fen07}.
\begin{de}
\label{de-3.1}
{\rm
The IFS $\{S_i\}_{i=1}^\ell$ is said to satisfy the
{\it asymptotically weak separation condition} (AWSC) if there exists a sequence $%
(t_n)$ of natural numbers such that
\begin{equation*}
\lim_{n\to \infty}\frac{1}{n}\log t_n=0
\end{equation*}
and for each $n\in {\Bbb N}$ and $\widetilde{Q}\in {\bf D}_n$ (see \eqref{e-dn} for the definition of ${\bf D}_n$),
\begin{equation}  \label{e-5.2}
\#\{S_u:\; u\in W_n, K_u\cap \widetilde{Q}\neq \emptyset\}\leq t_n.
\end{equation}
}
\end{de}
For instance, when $\beta>1$ is a Salem number, then an IFS $\{S_i\}_{i=1}^\ell$ on $\R$ satisfies the AWSC if each $S_i$ has the form
$$S_i(x)=\pm\beta^{-m_i}x+d_i,$$
where $m_i\in \N$ and $d_i\in \Z[\beta]$, here $\Z[\beta]$ denotes the integral ring generated by $\beta$.  For a proof, see  \cite[Proposition 5.3, Remark 5.5]{Fen07}.

\begin{rem}
\label{rem-LN}
{\rm The AWSC is strictly weaker than the WSC introduced in \cite{LaNg99}. To see it, for  $\beta\in (1,2)$ and  $m\in \N$,  set
$$Y^{\beta,m}:=\left\{\sum_{i=0}^n\epsilon_i\beta^i:\; n\in \N, \epsilon_i \in \{0,\pm 1,\ldots,\pm m\} \mbox{ for }0\leq i\leq n\right\}.$$
Erd\"{o}s and Komornik \cite{ErKo98} proved that if $\beta$  is not a Pisot number and $m\ge \beta-\beta^{-1}$, then $Y^{\beta,m}$ contains accumulation points. This implies that the IFS $\{\lambda x, \lambda x+1\}$ does not satisfies the WSC when $\lambda^{-1}\in (1, (\sqrt{5}+1)/2)$ is not a  Pisot number. However this IFS  satisfies the AWSC when $\lambda^{-1}$ is a Salem number; and there do exist infinitely many Salem numbers in $(1, (\sqrt{5}+1)/2)$ (see,  e.g.,  \cite{Bor02}).
}

\end{rem}
\subsection{Asymptotically good multifractal structure}
\label{Sub-2}
In this subsection, we assume that  $\{S_i\}_{i=1}^{\ell}$ is a $C^1$-conformal IFS on a compact set
$X\subset \R^d$ which satisfies the AWSC. Let $\nu$ be a self-conformal measure associated with $\{S_i\}_{i=1}^{\ell}$ and a probability vector $(p_1,\ldots, p_\ell)$. The main result of this subsection is the following.
\begin{thm}
\label{thm-3.1}
The measure $\nu$ has an asymptotically good multifractal structure over $\R_+$.
\end{thm}

To prove the above theorem, we need a simple lemma.
\begin{lem}
\label{lem-3.1}
Let $q>0$ so that $\tau'(q)$ exists and let $k\in \N$. Then there exist positive numbers $\epsilon, \delta, \gamma$  and $M$ (all depend on $q, k$)  with $\epsilon<\min\{1,q\}$, $\delta=\min\{1/(4k), 1/(4kq)\}$, and $\gamma<1/(4k)$, such that  for any $m\geq M$,
\begin{equation}
\label{e-3.1O}
\tau_m(q)\geq 2^{-m(\tau(q)+\gamma)},
\end{equation}
\begin{equation}
\label{e-3.1a}
\tau_m(q+\epsilon)~2^{m(\tau'(q)-\delta)\epsilon}\leq \tau_m(q)~2^{-m\gamma}
\end{equation}
and
\begin{equation}
\label{e-3.1b}
\tau_m(q-\epsilon)~2^{-m(\tau'(q)+\delta)\epsilon}\leq \tau_m(q)~2^{-m\gamma}.
\end{equation}
\end{lem}
\begin{proof}
Set $\delta=\min\{1/(4k), 1/(4kq)\}$. Since $\alpha=\tau'(q)$ exists, we can pick $0<\epsilon<\min\{1,q\}$ so that
$$
(\alpha-\delta/2)\epsilon\leq |\tau(q\pm \epsilon)-\tau(q)|\leq (\alpha+\delta/2)\epsilon.
$$
Set $\gamma=\min\{\epsilon\delta/8, 1/(4k)\}$. Since $\tau(u)=\lim_{n\to\infty}\tau_n(u)$ for each $u>0$, we take $M$ large enough such that for $m\geq M$,
$$
2^{-m(\tau(u)+\gamma)}\leq \tau_m(u)\leq 2^{-m(\tau(u)-\gamma)}\quad \mbox{for $u=q, \;q-\epsilon,\; q+\epsilon$}.
$$
Then we have
\begin{eqnarray*}
\tau_m(q+\epsilon)2^{m(\alpha-\delta)\epsilon} & \leq & 2^{-m(\tau(q+\epsilon)-\gamma)} 2^{m(\alpha-\delta)\epsilon}\\
&\leq &2^{-m(\tau(q)+\gamma)} 2^{-m(\tau(q+\epsilon)-\tau(q))} 2^{m((\alpha-\delta)\epsilon+2\gamma)}\\
&\leq &\tau_m(q)2^{-m(\alpha-\delta/2)\epsilon} 2^{m((\alpha-\delta)\epsilon+2\gamma)}\\
&\leq &\tau_m(q)2^{-m(\delta\epsilon/2-2\gamma)}\leq \tau_m(q)2^{-m\gamma},
\end{eqnarray*}
which proves (\ref{e-3.1a}). The proof of (\ref{e-3.1b}) is essentially identical.
\end{proof}

The following lemma is obvious.
\begin{lem}
\label{lem-2.2} Let $q>0$. For any $n\in {\Bbb N}$ and non-negative
numbers $x_{1},\ldots, x_{n}$,
\begin{equation}
\frac{1}{n}(x_{1}^{q}+\cdots +x_{n}^{q})\leq (x_{1}+\cdots +x_{n})^{q}\leq
n^{q}(x_{1}^{q}+\cdots +x_{n}^{q}).  \label{e-z4}
\end{equation}
\end{lem}
\bigskip

\begin{proof}[Proof of Theorem \ref{thm-3.1}] Set
 $$
 t_n=\max_{\widetilde{Q}\in {\bf D}_n}\#\{S_u:\; u\in W_n, \; K_u\cap \widetilde{Q}\neq \emptyset\},\qquad n\in \N.
 $$
 (See Sect.~\ref{Sub-1} for the  notation.)
Since the IFS $\{S_i\}_{i=1}^\ell$ is assumed to satisfy the AWSC (cf. Def.~\ref{e-3.1}), we have $$\lim_{n\to \infty}\frac{1}{n}\log t_n=0.$$
For each $n\in \N$, define an equivalence relation on $W_n$ by setting $u\sim v$ if and only if $S_u=S_v$. For $u\in W_n$, let $[u]$ denote the equivalence class containing $u$. In particular, we write
$$\overline{p}_{[u]}:=\sum_{v\in [u]}p_u,\quad  S_{[u]}:=S_u, \quad\mbox{and}\quad K_{[u]}:=K_u.$$
Iterating \eqref{e-1.2}, we obtain
\begin{equation}
\label{e-nu-new}
\nu=\sum_{[u]\in W_n/\sim}\overline{p}_{[u]}\; \nu\circ S_{[u]}^{-1}.
\end{equation}
Recall that by Proposition \ref{pro-3.1}, there is a sequence of positive numbers $(c_n)_{n=1}^\infty$ with $c_n>1$ and $\lim_{n\to \infty}(1/n)\log c_n=0$ such that
\eqref{e-3.1} holds.

From now on, we fix  $n\geq 0$  and $x\in \R$ such that $\mu(B_{2^{-n-1}}(x))>0$. Fix $q>0$ so that $\tau'(q)$ exists and fix $k\in \N$.
Let $\epsilon, \gamma, \delta, M$ be the positive numbers (depending on $q,k$)  given in Lemma \ref{lem-3.1} so that \eqref{e-3.1a}-\eqref{e-3.1b} hold.
Recall that we have the restrictions that
\begin{equation}
\label{e-z6}
\delta=\min\left\{\frac{1}{4k}, \frac{1}{4kq}\right\},\quad  \epsilon<\min\{1,q\}\quad  \mbox{and}\quad \gamma<\frac{1}{4k}.
\end{equation}
 Denote
$$
A=\frac{\nu(B_{2^{-n}}(x))}{\nu(B_{2^{-n-1}}(x))}.
$$
For convenience, denote $r=2^{-n}$. Let $n'$ be the unique integer satisfying
\begin{equation}\label{e-z0}
r/16<2^{-n'}\sqrt{d} \leq r/8.
\end{equation}
Clearly
\begin{equation}
\label{e-z7}
0<n'-n<4+\frac{\log d}{2\log 2}.
\end{equation}
A simple geometric argument shows that $B_r(x)$ intersects at most
$$
\Big(\frac{2r}{2^{-n'}}+1\Big)^d\leq (32\sqrt{d})^d
$$
elements in ${\bf D}_{n'}$. Hence we have
\begin{equation}
\label{e-z1}
\#\{[u]\in W_{n'}/\sim:\; K_{[u]}\cap B_r(x)\neq \emptyset\}\leq (32\sqrt{d})^d t_{n'}=:\tilde{t}_{n'}.
\end{equation}
Pick $[u_0]\in W_{n'}/\sim$ such that $K_{[u_0]}\cap B_{r/2}(x)\neq\emptyset$ and
$$
\overline{p}_{[u_0]}=\max\{\overline{p}_{[u]}:\; [u]\in W_{n'}/\sim, K_u\cap B_{r/2}(x)\neq \emptyset\}.
$$
By \eqref{e-nu-new},
$$
\sum_{[u]\in W_{n'}/\sim,\; K_{[u]}\cap B_{r/2}(x)\neq \emptyset}\overline{p}_{[u]}\geq \nu(B_{r/2}(x))=\frac{\nu(B_r(x))}{A}.
$$
Therefore we have
\begin{equation}
\label{e-z2}
\overline{p}_{[u_0]}\geq \frac{\nu(B_r(x))}
{\tilde{t}_{n'} A}.
\end{equation}
Set
$$\Gamma=\{[u]\in W_{n'}/\sim, K_{[u]}\cap B_{7r/8}(x)\neq \emptyset\}.$$
 By \eqref{e-z1}, $\# \Gamma\leq \tilde{t}_{n'}$.
Now define a measure $\eta$ on $\R^d$ by
$$
\eta=\sum_{[u]\in \Gamma}\overline{p}_{[u]}\;\nu\circ S_{[u]}^{-1}.
$$
Then by \eqref{e-nu-new}, the restrictions of $\eta$ and $\nu$ on $B_{7r/8}(x)$ coincide, i.e.,
$\eta|_{B_{7r/8}(x)}=\nu|_{B_{7r/8}(x)}$. By \eqref{e-z0},  $K_{[u]}\subset B_r(x)$ for all $[u]\in \Gamma$, hence by \eqref{e-nu-new},
\begin{equation}
\label{e-z3}
\sum_{[u]\in \Gamma}\overline{p}_{[u]}\leq \nu(B_r(x)).
\end{equation}

 Let $m'\in \N$. Denote
$$
\tau_{n'+m'}(F,q)=\sum_{Q\in {\bf D}_{n'+m'}:\; Q\subset F} \nu(Q)^q,\qquad F\subset \R^d.
$$
Since $K_{[u_0]}\cap B_{r/2}(x)\neq \emptyset$, by \eqref{e-z0}, for all those $Q\in {\bf D}_{n'+m'}$ with  $Q\cap K_{[u_0]}\neq \emptyset$, we have $Q\subset B_{3r/4}(x)$.  Hence we have
\begin{equation}
\label{e-tt}
\begin{split}
\tau_{n'+m'}(B_{3r/4}(x),q)&\geq \sum_{Q\in {\bf D}_{n'+m'}:\; Q\cap K_{[u_0]}\neq \emptyset}\nu(Q)^q\\
\mbox{}&\geq  \sum_{Q\in {\bf D}_{n'+m'},\; Q\cap K_{u_0}\neq \emptyset}(\overline{p}_{[u_0]})^q\;(\nu\circ S_{[u_0]}^{-1}(Q))^q\\
\mbox{}&=  (\overline{p}_{[u_0]})^q\; \sum_{Q\in {\bf D}_{n'+m'}}(\nu\circ S_{u_0}^{-1}(Q))^q\\
\mbox{}&\geq  (c_{n'})^{-(q+1)}(\overline{p}_{[u_0]})^q \tau_{m'}(q)\qquad \quad\mbox{(by \eqref{e-3.1})}\\
\mbox{}&\geq  (c_{n'}\tilde{t}_{n'}A)^{-q-1} \nu(B_r(x))^q   \tau_{m'}(q) \qquad\quad\mbox{(by \eqref{e-z2})}.
\end{split}
\end{equation}

On the other hand, we have
\begin{equation}
\label{e-tt1}
\begin{split}
\tau_{n'+m'}(B_{7r/8}(x),q) &=\sum_{Q\in {\bf D}_{n'+m'}:\; Q\subset B_{7r/8}(x)}\nu(Q)^q\\
\mbox{}&= \sum_{Q\in {\bf D}_{n'+m'}:\; Q\subset B_{7r/8}(x)}\eta(Q)^q
\leq \sum_{Q\in {\bf D}_{n'+m'}}\eta(Q)^q\\
\mbox{}&= \sum_{Q\in {\bf D}_{n'+m'}}\sum_{[u]\in \Gamma} \left(\overline{p}_{[u]}\;\nu\circ S_{[u]}^{-1}(Q)\right)^q\\
\mbox{}&\leq  \sum_{Q\in {\bf D}_{n'+m'}}(\tilde{t}_{n'})^q\sum_{[u]\in \Gamma} (\overline{p}_{[u]})^q\;\nu\circ S_{u}^{-1}(Q)^q\qquad \mbox{(by \eqref{e-z4})}\\
\mbox{}&\leq  (\tilde{t}_{n'})^q\sum_{[u]\in \Gamma} (\overline{p}_{[u]})^q\sum_{Q\in {\bf D}_{n'+m'}}\nu\circ S_{u}^{-1}(Q)^q\\
\mbox{}&\leq  (c_{n'}\tilde{t}_{n'})^{q+1}\nu(B_r(x))^q   \tau_{m'}(q)\qquad\mbox{(by \eqref{e-3.1},  \eqref{e-z3})}.
\end{split}
\end{equation}
Combining \eqref{e-tt} with \eqref{e-tt1} yields
\begin{equation}
\label{e-z10}
\tau_{n'+m'}(B_{7r/8}(x),q)\leq \tau_{n'+m'}(B_{3r/4}(x),q) \cdot (c_{n'} \tilde{t}_{n'}A)^{2q+2}, \quad \forall m'\in \N.
\end{equation}
We remark that in \eqref{e-tt}-\eqref{e-z10}, $q$ can be replaced by any positive number.

 From now on, assume that
 \begin{equation}
 \label{e-r1}
 \begin{split}
 m'\geq h_n=h_n(q,k):=&M+\frac{2q+3}{\gamma}\left(\log (4c_{n'}\tilde{t}_{n'}) +\log A \right.\\
 & \left. +\log (8^{1/q}\cdot 5^{d(q+1)/q}) \right),
 \end{split}
 \end{equation}
 where $\gamma$ and $M$ are the positive numbers given in Lemma \ref{lem-3.1} (they depend on $q$ and $k$).

It is easy to see that
$$
2^{2^{m'-1}}\geq (c_{n'}\tilde{t}_{n'}A)^{2q+2}.
$$
By \eqref{e-z10}, there exists $1\leq j\leq 2^{m'-1}$ such that
$$
\tau_{n'+m'}(B_{3r/4+j\cdot 2^{-(n'+m'-1)}\sqrt{d}}(x),q)\leq 2
\tau_{n'+m'}(B_{3r/4+(j-1)\cdot 2^{-(n'+m'-1)}\sqrt{d}}(x),q).
$$
(Otherwise,
\begin{eqnarray*}
\tau_{n'+m'}(B_{7r/8}(x),q)&\geq& \tau_{n'+m'}(B_{3r/4+2^{m'-1}\cdot 2^{-(n'+m'-1)}\sqrt{d}}(x),q)\\
&\geq& 2\tau_{n'+m'}(B_{3r/4+(2^{m'-1}-1)\cdot 2^{-(n'+m'-1)}\sqrt{d}}(x),q)\\
&\geq& \cdots\\
&\geq& 2^{2^{m'-1}} \tau_{n'+m'}(B_{3r/4}(x),q),
\end{eqnarray*}
which contradicts  \eqref{e-z10}.)
Fix such $j$ and take $$r'=3r/4+(j-1)\cdot 2^{-(n'+m'-1)}\sqrt{d}.$$
Then
\begin{equation}
\label{e-z5}
\tau_{n'+m'}(B_{r'+ 2^{-(n'+m'-1)}\sqrt{d}}(x),q)\leq 2
\tau_{n'+m'}(B_{r'}(x),q).
\end{equation}

Now define
\begin{equation*}
\begin{split}
{\mathcal F}&=\{ Q\in {\bf D}_{n'+m'}:\; Q\subset B_{7r/8}(x),\;
\nu(Q)<\nu(B_r(x))\cdot 2^{-m'(\alpha+\delta)} \},\\
{\mathcal F'}&=\{Q\in {\bf D}_{n'+m'}:\; Q\subset B_{7r/8}(x),\;
\nu(Q)>\nu(B_r(x))\cdot 2^{-m'(\alpha-\delta)} \}.
\end{split}
\end{equation*}
Then we have the estimation
\begin{eqnarray*}
\sum_{Q\in {\mathcal F}}\nu(Q)^q&\leq &\nu(B_r(x))^\epsilon2^{-m'(\alpha+\delta)\epsilon} \sum_{Q\in {\mathcal F}}\nu(Q)^{q-\epsilon}\\
&\leq &\nu(B_r(x))^\epsilon2^{-m'(\alpha+\delta)\epsilon}\tau_{n'+m'}(B_{7r/8}(x), q-\epsilon)\\
 &\leq&
(c_{n'}\tilde{t}_{n'})^{q-\epsilon+1}2^{-m'(\alpha+\delta)\epsilon} \nu(B_r(x))^q\tau_{m'}(q-\epsilon)\\
 &\mbox{}& \qquad\mbox{( by applying \eqref{e-tt1}, in which $q$ is replaced by $q-\epsilon$)}\\
  &\leq&
(c_{n'}\tilde{t}_{n'})^{q+2}\nu(B_r(x))^q\tau_{m'}(q)2^{-m'\gamma}\qquad\qquad\mbox{(by \eqref{e-3.1b})}\\
&\leq& (c_{n'}\tilde{t}_{n'}A)^{2q+3} 2^{-m'\gamma}\tau_{n'+m'}(B_{3r/4}(x), q)
 \qquad\mbox{(by \eqref{e-tt})}\\
&\leq& \frac{1}{4} \tau_{n'+m'}(B_{3r/4}(x), q)
\qquad \qquad\mbox{(by \eqref{e-r1})}.
 \end{eqnarray*}
Similarly, we have
\begin{eqnarray*}
\sum_{Q\in {\mathcal F}'}\nu(Q)^q&\leq &\nu(B_r(x))^{-\epsilon}\;2^{m'(\alpha-\delta)\epsilon} \sum_{Q\in {\mathcal F}'}\nu(Q)^{q+\epsilon}\\
&\leq &\nu(B_r(x))^{-\epsilon}\;2^{m'(\alpha-\delta)\epsilon}\;\tau_{n'+m'}(B_{7r/8}(x), q+\epsilon)\\
 &\leq&
 (c_{n'}\tilde{t}_{n'})^{q+\epsilon+1}\;2^{m'(\alpha-\delta)\epsilon} \nu(B_r(x))^q\tau_{m'}(q+\epsilon)\\
 &\mbox{}& \qquad\mbox{( by applying \eqref{e-tt1}, in which $q$ is replaced by $q+\epsilon$)}\\
  &\leq&
(c_{n'}\tilde{t}_{n'})^{q+2}\nu(B_r(x))^q\tau_{m'}(q)2^{-m'\gamma}\qquad\qquad\mbox{(by \eqref{e-3.1a})}\\
&\leq& (c_{n'}A\tilde{t}_{n'})^{2q+3} 2^{-m'\gamma}\tau_{n'+m'}(B_{3r/4}(x), q)
 \qquad\mbox{(by \eqref{e-tt})}\\
 &\leq& \frac{1}{4} \tau_{n'+m'}(B_{3r/4}(x), q).
\end{eqnarray*}

For any $Q\in {\bf D}_{n'+m'}$, we denote by
$$
Q^*=\prod_{s=1}^d\left[ \frac{i_s-2}{2^{n'+m'}},\frac{i_s+3}{2^{n'+m'}}\right)\quad \mbox{ if }\quad Q=\prod_{s=1}^d\left[ \frac{i_s}{2^{n'+m'}},\frac{i_s+1}{2^{n'+m'}}\right).
$$
Clearly, $Q^*$ contains exactly $5^d$ many elements in ${\bf D}_{n'+m'}$.  Set
\begin{equation*}
\begin{split}
T:&=8^{1/q}\cdot 5^{d(q+1)/q}\qquad \mbox{and}\\
{\mathcal F}''&=\{Q\in {\bf D}_{n'+m'}:\; Q\subset B_{r'}(x),\;
\nu(Q^*)>T\nu (Q)\}.
\end{split}
\end{equation*}
Then
\begin{eqnarray*}
\sum_{Q\in {\mathcal F}''}
\nu(Q)^q&\leq&
\sum_{Q\in {\bf D}_{n'+m'}: \; Q\subset B_{r'}(x)}
T^{-q}\nu(Q^*)^q\\
&\leq & T^{-q} 5^{d(q+1)} \tau_{n'+m'}(B_{r'+2^{-(n'+m'-1)}\sqrt{d}}(x),q)\qquad\mbox{(by \eqref{e-z4})}\\
&\leq & 2 \cdot T^{-q} 5^{d(q+1)} \tau_{n'+m'}(B_{r'}(x),q)\qquad\mbox{(by \eqref{e-z5})}\\
&=&\frac{1}{4}\tau_{n'+m'}(B_{r'}(x),q).
\end{eqnarray*}

Let
\begin{equation*}
\begin{split}
{\mathcal P}&=\left\{Q\in {\bf D}_{n'+m'}: \; Q\subset B_{r'}(x), \nu(Q^*)\leq  T\nu(Q) \mbox{ and }\right.\\
&\qquad \qquad  \, \; 2^{-m'(\alpha+\delta)}\leq {\nu(Q)}/{\nu(B_r(x))}\leq 2^{-m'(\alpha-\delta)}\}.
\end{split}
\end{equation*}
We have
\begin{eqnarray*}
\sum_{Q\in {\mathcal P}}
\nu(Q)^q &\geq& \sum_{Q\in {\bf D}_{n'+m'}: \; Q\subset B_{r'}(x)}\nu(Q)^q-\sum_{Q\in {\mathcal F}\cup {{\mathcal F}}'\cup {{\mathcal F}}''} \nu(Q)^q
\\
&=&
\tau_{n'+m'}(B_{r'}(x), q)-\sum_{Q\in {\mathcal F}\cup {{\mathcal F}}'\cup {{\mathcal F}}''} \nu(Q)^q\\
&\geq& \frac{1}{4}\tau_{n'+m'}(B_{r'}(x),q)\geq \frac{1}{4}\tau_{n'+m'}(B_{3r/4}(x),q)\\
&\geq& \frac14(c_{n'}\tilde{t}_{n'}A)^{-q-1}\nu(B_r(x))^q   \tau_{m'}(q) \qquad\qquad \mbox{(by \eqref{e-tt})}\\
&\geq & 2^{-m'/(4k)}\nu(B_r(x))^q  2^{-m'(\tau(q)+\gamma)}>0\qquad \mbox{(by \eqref{e-r1}, \eqref{e-3.1O})}.
\end{eqnarray*}
Clearly $\#{\mathcal P}\geq 1$.  Since $\nu(Q)\leq \nu(B_r(x))2^{-m'(\alpha-\delta)}$ for each $Q\in {\mathcal P}$, we have
\begin{eqnarray*}
\#{\mathcal P}&\geq& \nu(B_r(x))^{-q}2^{qm'(\alpha-\delta)}\sum_{Q\in {\mathcal P}}
\nu(Q)^q \\
&\geq& 2^{m'(\alpha q-\tau(q)-\delta q-\gamma-\frac{1}{4k})}
\geq 2^{m'(\alpha q-\tau(q)-\frac{3}{4k})} \qquad \mbox{(by \eqref{e-z6})}\\
&\geq & 5^d 2^{m(\alpha q-\tau(q)-\frac{1}{k})},
\end{eqnarray*}
with $m:=m'+n'-n$. Clearly $n+m=n'+m'$.

A simple geometric argument shows that there exists a family ${\mathcal P}'\subset {\mathcal P}$ with
$$\#{\mathcal P}'\geq 5^{-d}(\#{\mathcal P})\geq 2^{m(\alpha q-\tau(q)-\frac{1}{k})},$$  such that the set in $\{Q^*:\; Q\in {\mathcal P}'\}$ are disjoint.
Pick a large number  $C$ (independent of $n+m$) such that each $Q\in {\bf D}_{n+m}$ can be covered by $C$ many balls of radius of $2^{-n-m-1}$.
Then for any $Q\in {\mathcal P}'$, we can pick a ball $B_{2^{-n-m-1}}(y_Q))$ with $y_Q\in Q$ such that   $\nu(B_{2^{-n-m-1}}(y_Q))\geq \nu(Q)/C$.
Note that $Q\subset B_{2^{-n-m}}(y_Q)$ and $B_{2^{-n-m+1}}(y_Q)\subset Q^*$. We have
\begin{equation}
\label{e-g1}
\frac{ \nu(B_{2^{-n-m+1}}(y_Q))}{\nu(B_{2^{-n-m-1}}(y_Q))}\leq CT
\end{equation}
and
\begin{equation}\label{e-g2}
2^{-m(\alpha+\frac{1}{k})}\leq \frac{ \nu(Q)} {\nu(B_{2^{-n}}(x))} \leq \frac{ \nu(B_{2^{-n-m}}(y_Q))} {\nu(B_{2^{-n}}(x))}\leq \frac{ T\nu(Q)} {\nu(B_{2^{-n}}(x))}
\leq 2^{-m(\alpha-\frac{1}{k})}.
\end{equation}

Hence we have shown that when $n\geq 0$ and $x\in \R^d$ are given so that $\nu(B_{2^{-n-1}}(x))>0$,
for any $q\in \Omega_+$ and $k>0$,  if $m\geq h_n +n'-n$, where $h_n$ is defined as in \eqref{e-r1},
then there exist a disjoint family of balls $\{B_{2^{-n-m'}}(y_Q):\; Q\in {\mathcal P}'\}$ contained in $B_{2^{-n}}(x)$,
with $\#{\mathcal P}'\geq 2^{m(\alpha q-\tau(q)-\frac{1}{k})}$ and \eqref{e-g1}-\eqref{e-g2} hold.  This implies that $\nu$ has an asymptotically good multifractal structure on $\R_+$.
\end{proof}
\section{The proof of Theorem \ref{thm-1.2}}
\label{S-4}
We first give a simple lemma.

\begin{lem}
\label{lem-end}
Assume that $\mu$ is a self-similar measure associated with an IFS $\{S_i(x)=\rho x+a_i\}_{i=1}^\ell$ on $\R$ and a probability vector $(p_1,\ldots, p_\ell)$.
Let $K$ be the attractor of $\{S_i\}_{i=1}^\ell$. Then
we have the following properties.
\begin{itemize}
\item[(i)] If $\dim_HK=1$, then $\tau'_\mu(0+)\geq 1$.
\item[(ii)] If $p_i>\rho$ for some $1\leq i\leq \ell$, then $\tau_{\mu}^\prime(+\infty)\leq \log p_i/\log \rho<1$.
\end{itemize}
\end{lem}
\begin{proof}
To prove (i), assume that $\dim_HK=1$. Then it can be checked directly that $\tau_\mu(0)=-1$. Now let $0<q<1$. By the concavity of $x^q$ on $(0,+\infty)$, we have
$$
\sum_{Q\in {\bf D}_n}\mu(Q)^q=\sum_{Q\in {\bf D}_n:\; Q\cap K\neq \emptyset}\mu(Q)^q\leq v_n^{1-q},
$$
where $v_n=\#\{Q\in {\bf D}_n:\; Q\cap K\neq \emptyset\}$. Since  $v_n\leq c2^n$ for some constant $c>0$, we derive that  $\tau_\mu(q)\geq q-1$ and hence
$$\tau'_\mu(0+)=\lim_{q\to 0+}\frac{\tau_\mu(q)-\tau_\mu(0)} {q}\geq 1.$$
To show (ii), assume that $p_1>\rho$ without loss of generality. Then $\mu(S_1^n(K))\geq p_1^n$ for each $n\geq 1$, where $S_1^n$ denotes the $n$-th composition of $S_1$.
It follows that for $q>0$, $\Theta_\mu(q; \rho^n\mbox{diam}(K))\geq \mu(S_1^n(K))^q\geq p_1^{nq}$.  Hence $\tau_\mu(q)\leq q \log p_1/\log \rho$, which implies that
$\tau_{\mu}^\prime(+\infty)\leq \log p_1/\log \rho<1$.
\end{proof}

\begin{lem}
\label{lem-end1}
For $n\geq 4$, let  $\beta_n$ be the largest real root of the polynomial $Q_n(x)=x^n-x^{n-1}-\cdots-x+1$. Then  $\beta_n^{n+1}>2^n$ for $n\geq 5$.
\end{lem}
\begin{proof}
Multiplying $x-1$ by $Q_n(x)$ yields
$$
(x-1)Q_n(x)=x^{n+1}-2x^n+2x-1.
$$
Hence $(2-\beta_n)\beta_n^n=2\beta_n-1$. Now assume that $n\geq 5$. It is easy to check that $\beta_n>1.8$. Hence $2-\beta_n=\frac{2\beta_n-1}{\beta_n^n}<3\times 1.8^{-n}$. Let $\epsilon_n=2-\beta_n$. Then $(n+1)\epsilon_n\leq (n+1)\times 3\times 1.8^{-n}<1$. By the Mean Value Theorem,
$$(2-\epsilon_n)^{n+1}=2^{n+1}-(n+1)\epsilon_n \xi_n^n\geq 2^{n+1}-(n+1)\epsilon_n 2^n>2^n.$$ That is, $\beta_n^{n+1}>2^n$.
\end{proof}

\begin{proof}[Proof of Theorem \ref{thm-1.2}]
Assume $\lambda=\beta_n^{-1}$, $n\geq 4$. Iterate \eqref{1.1} $k$-times to get
\begin{equation}
\label{e-en1}
\nu_\lambda=\sum_{I\in {\mathcal A}^k}\frac{1}{2^k}\nu_\lambda\circ S_I^{-1},
\end{equation}
where ${\mathcal A}=\{1,2\}$. Define an equivalence relation $\sim$ on ${\mathcal A}^k$ $I\sim J$ if and only if $S_I=S_J$.  For $I\in {\mathcal A}^k$, let $I$ denote the equivalence class that contains $I$. Then
\eqref{e-en1} can be rewritten as
\begin{equation}
\label{e-en2}
\nu_\lambda=\sum_{[I]\in {\mathcal A}^k/\sim}\frac{\#[I]}{2^k}\nu_\lambda\circ S_{[I]}^{-1},
\end{equation}
where $\#[I]$ denotes the cardinality of  the equivalence class $[I]$. To prove Theorem \ref{thm-1.2}, according to Lemma \ref{lem-end}, it suffices to show that there exists $k\in \N$ and $I\in {\mathcal A}^k$ such that $\frac{\#[I]}{2^k}>\lambda^k$. We prove this fact by considering two different cases separately: $n\geq 5$ and $n=4$. In the first case, we take $k=n+1$ and $I=1\underbrace{2\cdots 2}_{n-1}1$.  It is easy to see that  $I\sim 2\underbrace{1\cdots 1}_{n-1}2$, and hence $\#[I]\geq 2$. Then the inequality $\frac{\#[I]}{2^k}>\lambda^k$ follows from Lemma \ref{lem-end1}. Next we consider the case  $n=4$. Take $k=15$ and let $$I=122211121112221.$$
A direct computation shows that $\#[I]=10$ (see Table \ref{table-1}) and  $\frac{\#[I]}{2^k}>\lambda^k$.
\begin{table}
\centering
\caption{Elements in $[I]$}
\vspace{0.05 in}
\begin{footnotesize}
\begin{raggedright}
\begin{tabular}{p{1.3 in} p{1. in} }
\hline \rule{0pt}{3ex}
%& Effect of increasing \\
%Definition & substrate stiffness \\[3pt]
%\hline \rule{0pt}{3ex}
\!\!$122122122211112$ & $122122211112221$\\
$122122211121112$ & $122211112221221$\\
$122211121112221$ & $122211121121112$\\
$211112221221221$ & $211121112221221$\\
$211121121112221$ & $211121121121112$\\
\hline
\end{tabular}
\label{table-1}
\end{raggedright}
\end{footnotesize}
\end{table}
\end{proof}

\section{Absolutely continuous self-similar measures with non-trivial range of local dimensions}
\label{S-5}
In this section, we show the existence of an absolutely continuous  self-similar measure on $\R$ with non-trivial range of local dimensions. Indeed, we have the following result.

\begin{pro} For $\lambda, u\in (0,1)$, let $\Phi_{\lambda, u}:=\{S_i\}_{i=1}^3$ be the IFS on $\R$ given by
$$S_1(x)=\lambda x,\quad S_2(x)=\lambda x+u,\quad \quad S_3(x)=\lambda x+1.$$
Let $\mu_{\lambda,u}$ be the self-similar measure associated with  $\Phi_{\lambda, u}$ and the probability vector $\{1/4, 5/12, 1/3\}$, i.e.,
$\mu=\mu_{\lambda,u}$ satisfies $$\mu=\frac{1}{4}\mu\circ S_1^{-1} +\frac{5}{12}\mu\circ S_2^{-1}+\frac{1}{3}\mu\circ S_3^{-1}.$$
Then for ${\mathcal L}^2$-a.e. $(\lambda, u)\in (0.3405, 0.3439)\times (1/3, 1/2)$, $\mu_{\lambda,u}$ is absolutely continuous, and  the range of local dimensions of $\mu_{\lambda, u}$ contains a non-degenerate interval, on which the multifractal formalism for $\mu_{\lambda, u}$ is valid.
\label{pro-5}
\end{pro}

\begin{proof} For $q>0$, let $\tau(q, \lambda, u)$ denote the $L^q$ spectrum of $\mu_{\lambda, u}$. Applying Theorem 6.2 by Falconer in \cite{Fal99},  for each $0<\lambda<1/2$,  we have for ${\mathcal L}$-a.e. $u\in (0,1)$,
$$
\tau(q, \lambda, u)=\min\left\{\frac{\log \left((1/4)^q+(5/12)^q+(1/3)^q\right)}{\log \lambda},\;q-1 \right\},\quad 1<q<2.
$$
Write $f(q)= (1/4)^q+(5/12)^q+(1/3)^q$. Clearly $f(1)=1$. It is easily checked that $\log f(q)$ is strictly convex over $q>0$ and hence $\frac{\log f(q)}{q-1}$ is strictly increasing over $q>1$. Note that $f(1.5)^{1/(1.5-1)}=f(1.5)^2\approx 0.34387$. Hence for $0<\lambda<0.3438$ and $q>1.5$,
$$
g(q, \lambda):=\frac{\log \left((1/4)^q+(5/12)^q+(1/3)^q\right)}{\log \lambda}<q-1.
$$
Therefore for every  $0<\lambda<0.3438$, we have for ${\mathcal L}$-a.e. $u\in (0,1)$,  $\tau(q, \lambda, u)=g(q, \lambda)$ for $1.5<q<2$; clearly, $g$ is differentiable in $q$, thus by Theorem 1.1 in \cite{Fen07}, the range of local dimensions of $\mu_{\lambda, u}$ contains the non-degenerate interval $\{\frac{dg(q, \lambda)}{dq}:\; 1.5<q<2\}$, on which the multifractal formalism for $\mu_{\lambda, u}$ is valid.

To complete the proof of the proposition, it suffices  to show that for every $u\in (1/3, 1/2)$, $\mu_{\lambda, u}$ is absolutely continuous for $\mathcal L$-a.e.
$\lambda\in (0.3405, 0.3438)$.  This is done by simply applying a general  result by Peres and Solomyak (see Theorem 1.3 in \cite{PeSo98}).  The transversality condition needed there holds since  $\lambda (\sqrt{3}+1)<1$ (see the remark after Theorem 1.3 in \cite{PeSo98}) and $0.3405>(1/4)^{1/4}(5/12)^{5/12}(1/3)^{1/3}\approx 0.34042$.
\end{proof}
\medskip

We end the paper by posing the following  unsolved questions:
\begin{itemize}
\item[(i)] Does  Theorem \ref{thm-1.1} hold for all $\lambda\in (1/2, 1)$? Moreover, does Theorem \ref{thm-1.3} hold for all self-conformal measures?
\item[(ii)]
 Is it always true that  $\tau^\prime_{\ber}(+\infty)<1$ when $\lambda^{-1}$ is a Salem number?\end{itemize}

 We remark that the inequality in (ii) always holds  in the case that  $\lambda^{-1}$ is a Pisot number in $(1,2)$;  because in the Pisot case, $\tau^\prime_{\ber}(1)=\dim_H\ber<1$ (cf. \cite{Fen03}), hence $\tau^\prime_{\ber}(+\infty)\leq \tau^\prime_{\ber}(1)<1$.

\noindent {\bf Acknowledgements}.  The author was partially
supported by the RGC grant and the Focused Investments Scheme  in CUHK.  He is grateful to Boris Solomyak for his kind question in \cite{Sol07} which motivates this research. He thanks  Tsz-Chiu Kwok for providing some numerical computations and Table \ref{table-1}.

\end{document}